\documentclass[a4,10pt]{amsart}
\usepackage{amssymb,amstext,amsmath,amscd,amsthm,amsfonts,enumerate,graphicx,latexsym,cite,mathptmx}
\usepackage[usenames]{color}
\newtheorem{thm}{Theorem}[section]
\newtheorem{cor}[thm]{Corollary}
\newtheorem{lem}[thm]{Lemma}
\newtheorem{prop}[thm]{Proposition}
\theoremstyle{definition}

\newtheorem{rem}[thm]{Remark}

\newtheorem{conj}[thm]{Conjecture}
\newtheorem{ex}[thm]{Example}

\newtheorem*{claim*}{Claim}
\theoremstyle{remark}

\numberwithin{equation}{thm}

\def\spec{\operatorname{Spec}}
\def\Ext{\operatorname{Ext}}
\def\Hom{\operatorname{Hom}}
\def\Tor{\operatorname{Tor}}
\def\Tr{\operatorname{Tr}}
\def\grade{\operatorname{grade}}
\def\depth{\operatorname{depth}}
\def\image{\operatorname{Im}}
\def\lhom{\operatorname{\underline{Hom}}}

\def\S{\mathrm{S}}
\def\p{\mathfrak{p}}
\def\Ass{\operatorname{Ass}}
\def\Min{\operatorname{Min}}
\def\gdim{\operatorname{Gdim}}

\def\syz{\mathrm{\Omega}}
\def\xx{\boldsymbol{x}}

\def\pd{\operatorname{pd}}

\def\cext{\mathrm{\widehat{Ext}}\mathrm{}}

\def\Z{\mathbb{Z}}
\def\E{\mathrm{E}}
\def\m{\mathfrak{m}}
\def\supp{\operatorname{Supp}}
\def\X{\operatorname{X}}
\def\Y{\operatorname{Y}}
\def\Ker{\operatorname{Ker}}

\def\Cok{\operatorname{Coker}}
\def\Q{\operatorname{Q}}
\def\NF{\operatorname{NF}}
\def\assh{\operatorname{Assh}}
\def\height{\operatorname{ht}}
\def\q{\mathfrak{q}}
\begin{document}
\allowdisplaybreaks
\title{Two generalizations of Auslander--Reiten duality and applications}
\author{Arash Sadeghi}
\address{School of Mathematics, Institute for Research in Fundamental Sciences (IPM), P.O. Box: 19395-5746, Tehran, Iran}
\email{sadeghiarash61@gmail.com}
\author{Ryo Takahashi}
\address{Graduate School of Mathematics, Nagoya University, Furocho, Chikusaku, Nagoya, Aichi 464-8602, Japan/Department of Mathematics, University of Kansas, Lawrence, KS 66045-7523, USA}
\email{takahashi@math.nagoya-u.ac.jp}
\urladdr{https://www.math.nagoya-u.ac.jp/~takahashi/}
\dedicatory{Dedicated to Professor Mohammad T. Dibaei on the occasion of his retirement}
\thanks{Sadeghi's research was supported by a grant from IPM. Takahashi was partly supported by JSPS Grant-in-Aid for Scientific Research 16K05098 and JSPS Fund for the Promotion of Joint International Research 16KK0099}
%%%%%%%%%%%%%%%%%%%%%%%%%%%%%%%%%%%%%%%%%%%%%%%%%%%%%%%%%%%%%%%%
\begin{abstract}
This paper extends Auslander--Reiten duality in two directions.
As an application, we obtain various criteria for freeness of modules over local rings in terms of vanishing of Ext modules, which recover a lot of known results on the Auslander--Reiten conjecture.
\end{abstract}
\maketitle
%%%%%%%%%%%%%%%%%%%%%%%%%%%%%%%%%%%%%%%%%%%%%%%%%%%%%%%%%%%%%%%%
\section{Introduction}

The {\em Auslander--Reiten conjecture} \cite{AuRe} is one of the most celebrated conjectures in the representation theory of algebras.
This long-standing conjecture is known to hold true over several classes of algebras, including algebras of finite representation type \cite{AuRe} and symmetric artin algebras with radical cube zero \cite{Ho}.
This conjecture is closely related to other conjectures such as the {\em Nakayama conjecture} \cite{Nak,AuRe} and the {\em Tachikawa conjecture} \cite{ABS,Tach}.
Although the Auslander--Reiten conjecture was initially proposed over artin algebras, it remains meaningful for arbitrary commutative noetherian rings:

\begin{conj}[Auslander--Reiten]\label{ar}
Let $R$ be a commutative noetherian ring $R$ and let $M$ be a finitely generated $R$-module.
If $\Ext^{i}_{R}(M,M )=0=\Ext^{i}_{R}(M,R)$ for all $i\geq 1$, then $M$ is projective.
\end{conj}

\noindent
Auslander, Ding and Solberg \cite{ADS} proved that Conjecture \ref{ar} holds for any complete intersection local rings.
Recently there has been various progress towards Conjecture \ref{ar}; see \cite{ACST,CIST,CT,CH1,CH2,DEL,GT, HL, CSV, OnYos} for instance.
Using Auslander--Reiten duality, Araya \cite{A} proved that if all Gorenstein local rings of dimension at most one satisfy the Auslander--Reiten conjecture, then so do all Gorenstein local rings.

In this paper, we extend Auslander--Reiten duality in two directions; the following two theorems are included in the main results of this paper.

\begin{thm}\label{s}
Let $R$ be a commutative noetherian ring.
Let $n\ge1$ be an integer.
Let $M,N$ be finitely generated $R$-modules.
Assume that $M$ is $(n+1)$-torsionfree (e.g. totally reflexive) and $\NF(M)\cap\NF(N)\subseteq\Y^n(R)$.
\begin{enumerate}[\rm(1)]
\item
There exists an exact sequence of $R$-modules
\begin{align*}
0\to&\Ext^{n-1}_R(\Hom_R(M,N),R)\to  \Ext^{n-1}_R(N,M) \to \Ext^n_R(\lhom_R(M,N),R)\\
\to& \Ext^n_R(\Hom_R(M,N),R).
\end{align*}
\item
For all integers $i\le n-2$ one has an isomorphism of $R$-modules
$$
\Ext^i_R(\Hom_R(M,N),R)\cong\Ext^i_R(N,M).
$$
\end{enumerate}
\end{thm}

\begin{thm}\label{t}
Let $(R,\m,k)$ be a Cohen--Macaulay local ring of dimension $d\ge2$ with canonical module $\omega$.
Let $M,N$ be maximal Cohen--Macaulay $R$-modules such that $\NF(M)\cap\NF(N)\subseteq\{\m\}$.
\begin{enumerate}[\rm(1)]
\item
One has an isomorphism of $R$-modules
$$
\lhom_R(M,N)^\vee\cong\Ext_R^{d-1}(M^*,N^\dag).
$$
\item
Suppose that $M$ is locally totally reflexive (e.g. locally free) on the punctured spectrum of $R$.
Then for all $1\le i\le d-2$ one has an isomorphism of $R$-modules
$$
\Ext_R^i(M,N)^\vee\cong\Ext_R^{(d-1)-i}(M^*,N^\dag).
$$
\end{enumerate}
\end{thm}

\noindent
Let us explain the notation and terminology used above.
For an integer $n$, we denote by $\Y^n(R)$ the set of prime ideals $\p$ with $\depth R_\p\ge n$.
An $n$-torsionfree module is a module $M$ with $\Ext_R^i(\Tr M,R)=0$ for $1\le i\le n$, where $\Tr M$ stands for the Auslander transpose of $M$.
We denote by $\NF(M)$ the non-free locus of an $R$-module $M$, that is, the set of prime ideals $\p$ such that $M_\p$ is not $R_\p$-free.
For $R$-modules $M,N$ we denote by $\lhom_R(M,N)$ the stable Hom module, namely, the quotient of $\Hom_R(M,N)$ by the homomorphisms factoring through projective modules.
Also, $(-)^\vee=\Hom_R(-,\E_R(k))$, $(-)^*=\Hom_R(-,R)$, and $(-)^\dag=\Hom_R(-,\omega)$ are the Matlis, algebraic and canonical duals, respectively.

Both of the above two theorems recover the following celebrated Auslander--Reiten duality theorem for $d\ge2$.
In fact, Theorem \ref{t} is included in a more general result (Theorem \ref{14}), which recovers Corollary \ref{ard} for all $d\ge0$.

\begin{cor}[Auslander--Reiten duality]\label{ard}
Let $(R,\m)$ be a $d$-dimensional Gorenstein local ring.
Let $M,N$ be maximal Cohen--Macaulay $R$-modules such that $\NF(M)\cap\NF(N)\subseteq\{\m\}$.
Then for all $i\in\Z$ there is an isomorphism of Tate cohomology modules
$$
\cext_R^i(N,M)^\vee\cong\cext_R^{(d-1)-i}(M,N).
$$
\end{cor}

\noindent
There are also many other applications of the above theorems, including an improved version of a main theorem of Celikbas and Takahashi \cite[Theroem 1.1]{CT}, and the following result.
Recall that a commutative noetherian ring is called {\em generically Gorenstein} if $R_\p$ is Gorenstein for all $\p\in\Ass R$.

\begin{cor}\label{cc}
Let $R$ be a generically Gorenstein local ring of depth $t$, and let $2\le n\le t$ be an integer.
Let $M$ be a finitely generated $R$-module such that $\NF(M)\subseteq\Y^n(R)$.
Then $M$ is free if it satisfies the following three conditions.
\begin{enumerate}[\rm(1)]
\item
$\Ext^i_R(M,M)=0$ for $n-1\leq i\leq t-1$.
\item
$\Ext^i_R(M,R)=0$ for $i>0$.
\item $\Ext^i_R(\Hom_R(M,M),R)=0$ for $n\leq i\leq t$, or $\Hom_R(M,M)$ has finite G-dimension.
\end{enumerate}
\end{cor}

\noindent
It is worth mentioning that Corollary \ref{cc} simultaneously generalizes several known results on the Auslander--Reiten conjecture, including the theorems of Araya \cite[Corollary 4]{A}, Ono and Yoshino \cite[Theorem]{OnYos} and Araya, Celikbas, Sadeghi and Takahashi \cite[Corollary 1.6]{ACST}, and substantial parts of the theorems of Huneke and Leuschke \cite[Theorem 1.3]{HL}, Goto and Takahashi \cite[Theorem 1.5(2)]{GT} and Dao, Eghbali and Lyle \cite[Theorem 3.16]{DEL}.
Furthermore, Corollary \ref{cc} yields the following result.

\begin{cor}\label{abc}
Let $R$ be a noetherian normal local ring of depth $t$.
A finitely generated $R$-module $M$ is free, if either of the following conditions holds.
\begin{enumerate}[\rm(1)]
\item
$\Ext^i_R(M,M)=\Ext^j_R(M,R)=\Ext^h_R(\Hom_R(M,M),R)=0$ for all $1\le i\le t-1$, $j\ge1$ and $2\le h\le t$.
\item
$\Ext^i_R(M,M)=\Ext^j_R(M,R)=0$ for all $1\le i\le t-1$ and $j\ge1$, and that $\Hom_R(M,M)$ has finite G-dimension.
\end{enumerate}
\end{cor}

In Section 2, we extend Auslander--Reiten duality.
We prove our main theorems in this section that yield Theorems \ref{s} and \ref{t} as special cases, and recover Corollary \ref{ard}.
In Section 3, we apply our theorems to give various criteria for freeness of modules in relation to Conjecture \ref{ar}, including Corollaries \ref{cc} and \ref{abc}.
Among other things, we generalize the main theorem of \cite{ACST} in this section (Corollary \ref{t2}).

%%%%%%%%%%%%%%%%%%%%%%%%%%%%%%%%%%%%%%%%%%%%%%%%%%%%%%%%%%%%%%%%%
\section*{Convention}
Throughout this paper, we assume that all rings are commutative noetherian and all modules are finitely generated.
Let $R$ be a ring.
We denote by $(-)^*$ the $R$-dual $\Hom_R(-,R)$.
If $R$ is local, then $\m,k,d,t,(-)^\vee$ always stand for the maximal ideal of $R$, the residue field of $R$, the (Krull) dimension of $R$, the depth of $R$, and the Matlis dual $\Hom_R(-,\E_R(k))$, respectively.
Whenever $R$ admits a canonical module $\omega$, we denote by $(-)^\dag$ the canonical dual $\Hom_R(-,\omega)$.
For an $R$-module $M$, we denote by $\syz M$ and $\Tr M$ the (first) syzygy and the (Auslander) transpose of $M$.
We refer the reader to \cite{AB} for details on syzygies, transposes, $n$-torsionfree modules, totally reflexive modules and G-dimension.
For the definition and basic properties of Tate cohomology, we refer the reader to \cite[Section 7]{catgp}.
We say that an $R$-module $M$ is locally free (resp. totally reflexive, of finite G-dimension) on a subset $A$ of $\spec R$, provided that $M_\p$ is free (resp. totally reflexive, of finite G-dimension) as an $R_\p$-module for all $\p\in A$.
For an integer $n$ and an $R$-module $K$, let $\X^n(K)$ (resp. $\Y^n(K)$) denote the set of prime ideals $\p$ of $R$ such that $\depth_{R_\p}K_\p$ is at most (resp. at least) $n$.
%We refer the reader to \cite{G} for details on G$_K$-dimension.
%%%%%%%%%%%%%%%%%%%%%%%%%%%%%%%%%%%%%%%%%%%%%%%%%%%%%%%%%%%%%%%%%
\section{Extending Auslander--Reiten duality to $n$-torsionfree modules}

In this section, we extend Auslander--Reiten duality to $(n+1)$-torsionfree modules for $n\ge1$.
First we present some examples of such modules.

\begin{ex}\label{15}
Let $n\ge1$ be an integer.
An $R$-module $M$ is $(n+1)$-torsionfree in each of the following cases.
\begin{enumerate}[(1)]
\item
$M$ is totally reflexive.
\item
$R$ is local, $M$ is locally totally reflexive on the punctured spectrum, $\depth M\ge n\ge2$ and $\Ext_R^{n-1}(M^*,R)=0$.
\item
$M$ is locally of finite G-dimension on $\X^{n-1}(R)$ and $M$ is a $(n+1)$st syzygy.
\item
$R$ is Cohen--Macaulay and local, $M$ is locally of finite G-dimension on the punctured spectrum, and $M$ is a syzygy of a maximal Cohen--Macaulay module.
\end{enumerate}
\end{ex}

\begin{proof}
(1) By the definition of total reflexivity $\Ext^{>0}_{R}(\Tr M,R)=0$.
Thus $M$ is $i$-torsionfree for all $i\ge0$.

(2) We see $\depth_{R_\p}M_\p\ge\min\{n,\depth R_\p\}$ for $\p\in\spec R$.
The assumption $\depth_RM\ge n$ especially says $n\le d$, and $M_\p$ is totally reflexive for $\p\in\X^{n-1}(R)$.
Hence $M$ is $n$-torsionfree by \cite[Proposition 2.4]{DS}.
As $n\ge2$, we have $\Ext^{n+1}_{R}(\Tr M,R)\cong\Ext^{n-1}_{R}(M^*,R)=0$.
Therefore $M$ is $(n+1)$-torsionfree.

(3) This is a consequence of \cite[Theorem 43]{M}.

(4) Write $M=\syz N$ with $N$ maximal Cohen--Macaulay.
Note that $N$ is also locally of finite G-dimension on the punctured spectrum.
By \cite[Proposition 2.4]{DS} the module $N$ is a $d$th syzygy, and hence $M$ is a $(d+1)$st syzygy.
The assertion now follows from (3).
\end{proof}

To prove our first main result and some other results given later, we establish a lemma.

\begin{lem}\label{21}
Let $n\ge1$ be an integer.
Let $M,N,K$ be $R$-modules such that $\NF(M)\cap\NF(N)\subseteq\Y^n(K)$. \begin{enumerate}[\rm(1)]
\item
There exists an exact sequence of $R$-modules
$$
0 \to \Ext^{n-1}_R(\Hom_R(M,N),K) \to \Ext^{n-1}_R(M^*\otimes_RN,K) \to \Ext^n_R(\lhom_R(M,N),K)\to \Ext^n_R(\Hom_R(M,N),K).
$$
\item
For all integers $i\le n-2$ one has an isomorphism of $R$-modules $\Ext^i_R(\Hom_R(M,N),K) \cong \Ext^i_R(M^*\otimes_RN,K)$.
\end{enumerate}
\end{lem}

\begin{proof}
According to \cite[Theorem (2.8)]{AB}, there is an exact sequence\footnote{In fact, one can add ``$0\to$'' at the beginning of the exact sequence, that is, $\Ker f\cong\Tor_2^R(\Tr M,N)$.}
$$
\Tor_2^R(\Tr M,N) \to M^*\otimes_R N \xrightarrow{f} \Hom_R(M,N) \to \Tor_1^R(\Tr M,N) \to 0.
$$
Set $H=\Ker f$, $C=\Tor_1^R(\Tr M,N)$ and $L=\image f$.
Note that the support of $H$ is contained in that of $\Tor_2^R(\Tr M,N)$.
Since $\NF(M)\cap\NF(N)\subseteq\Y^n(K)$, we see that $\Ext^i_R(C,K)=0=\Ext^i_R(H,K)$ for all $i<n$ by \cite[Lemma (4.5)]{AB}.
Dualizing the above exact sequence by $K$ yields isomorphisms $\Ext^i_R(\Hom_R(M,N),K)\cong\Ext^i_R(L,K)$ and $\Ext_R^j(L,K)\cong\Ext_R^j(M^*\otimes_RN,K)$ for all $i\le n-2$ and $j\le n-1$, and an exact sequence 
$$
0 \to \Ext^{n-1}_R(\Hom_R(M,N),K) \to \Ext^{n-1}_R(L,K)\to \Ext^n_R(C,K)\to \Ext^n_R(\Hom_R(M,N),K).
$$
It remains to note from \cite[Lemma (3.9)]{Yo} that $\lhom_R(M,N)\cong\Tor_1^R(\Tr M,N)=C$.
\end{proof}

Let $K$ be an $R$-module.
An $R$-module $M$ is called {\em $n$-torsionfree with respect to $K$} provided that $\Ext^i_R(\Tr M,K)=0$ for all $1\leq i\leq n$.
Note that $M$ is $n$-torsionfree if and only if it is $n$-torsionfree with respect to $R$.
%When $K$ is a semidualing module and $M$ has finite projective dimension, $M$ is $n$-torisonfree if and only if it is $n$-torsionfree with respect to $K$ \cite[Theorem 4.1]{Sa}.
For $R$-modules $M,N$ we denote by $\grade(M,N)$ the smallest non-negative integer $n$ such that $\Ext^i_R(M,N)\neq0$.
The following theorem is the first main result of this paper, which includes Theorem \ref{s} from the Introduction.

\begin{thm}\label{7}
Let $n\ge1$ be an integer.
Let $M,N,K$ be $R$-modules.
Assume that $M$ is $(n+1)$-torsionfree with respect to $K$, and that $\NF(M)\cap\NF(N)\subseteq\Y^n(K)$.
Then the following hold.
\begin{enumerate}[\rm(1)]
\item
There exists an exact sequence of $R$-modules
\begin{align*}
0 \to& \Ext^{n-1}_R(\Hom_R(M,N),K) \to \Ext^{n-1}_R(N,M\otimes_RK) \to \Ext^n_R(\lhom_R(M,N),K)\\
\to& \Ext^n_R(\Hom_R(M,N),K).
\end{align*}
\item
For all integers $i\le n-2$ one has an isomorphism of $R$-modules
$$
\Ext^i_R(\Hom_R(M,N),K)\cong\Ext^i_R(N,M\otimes_RK).
$$
\end{enumerate}
\end{thm}

\begin{proof}
Thanks to Lemma \ref{21}, it suffices to show that $\Ext^i_R(M^*\otimes_R N,K)\cong\Ext_R^i(N,M\otimes_RK)$ for all integers $i\le n-1$.
There are two spectral sequences converging to the same point:
$$
E_2^{pq}=\Ext^p_R(\Tor_q^R(M^*,N),K) \Longrightarrow H^{p+q},\qquad
F_2^{pq}=\Ext^p_R(N,\Ext^q_R(M^*,K)) \Longrightarrow H^{p+q}.
$$
As $\grade(\Tor_i^R(M^*,N),K)\ge n$ for all $i>0$ by \cite[Lemma (4.5)]{AB}, we have that $E_2^{pq}=0$ if $q>0$ and $p<n$, and obtain an isomorphism $E_2^{i0}\cong H^i$ for all integers $i\le n$.
Since $\Ext^i_R(\Tr M,K)=0$ for all $1\le i\le n+1$, we have $\Ext^i_R(M^*,K)=0$ for all $1\le i\le n-1$.
Hence $F_2^{pq}=0$ if $1\le q\le n-1$, and we get $F_2^{i0}\cong H^i$ for all integers $i\le n-1$.
As $n+1\ge2$, the module $M$ is $2$-torsionfree with respect to $K$, and $M\otimes_RK\cong\Hom_R(M^*,K)$ by \cite[Proposition (2.6)]{AB}.
Thus there are isomorphisms
$$
\Ext^i_R(M^*\otimes_R N,K)=E_2^{i0}\cong H^i\cong F_2^{i0}=\Ext^i_R(N,\Hom_R(M^*,K))\cong\Ext^i_R(N,M\otimes_RK)
$$
for all integers $i\le n-1$ and the proof of the theorem is completed.
\end{proof}

We obtain the following corollary as a consequence of the above theorem.

\begin{cor}\label{12}
Let $R$ be a local ring, and let $M,N$ be $R$-modules.
Assume that $M$ is $(t+1)$-torsionfree and that $\NF(M)\cap\NF(N)\subseteq\Y^t(R)$.
\begin{enumerate}[\rm(1)]
\item
If $\Ext_R^t(\Hom_R(M,N),R)=\Ext_R^{t-1}(N,M)=0$, then $\lhom_R(M,N)=0$.
\item
Suppose that $\Hom(M,N)$ has depth at least $2$ and finite G-dimension.
Then there exists an isomorphism
$$
\Ext_R^{t-1}(N,M) \cong \Ext_R^t(\lhom_R(M,N),R).
$$
In particular, $\Ext_R^{t-1}(N,M)=0$ if and only if $\lhom_R(M,N)=0$.
\end{enumerate}
\end{cor}

\begin{proof}
(1) If $t=0$, then $\Hom_R(\Hom_R(M,N),R)=0$, and $\Hom_R(M,N)=0$ by \cite[Exercise 1.2.27]{BH}, which implies $\lhom_R(M,N)=0$.
Let $t>0$.
Theorem \ref{7}(1) shows $\Ext^t_R(\lhom_R(M,N),R)=0$, while $\grade_R\lhom_R(M,N)\ge t$ by \cite[Lemma (4.5)]{AB} or \cite[Proposition 1.2.10(a)]{BH}.
Hence $\grade_R\lhom_R(M,N)>t$.
Thus, if $\lhom_R(M,N)$ is nonzero, then its annihilator contains an $R$-sequence of length more than $t=\depth R$, which cannot occur.
We must have $\lhom_R(M,N)=0$.

(2) The last assertion is shown similarly to (1).
As for the first assertion, in view of Theorem \ref{7}(1), it is enough to show that $\Ext^i_R(\Hom_R(M,N),R)=0$ for $i=t-1,t$.
We have $\gdim_R(\Hom_R(M,N))=\depth R-\depth_R(\Hom_R(M,N))\le t-2$, which implies that $\Ext^i_R(\Hom_R(M,N),R)=0$ for all $i>t-2$.
\end{proof}

\begin{rem}
In fact, one can generalize Corollary \ref{12} to a $(t+1)$-torsionfree module $M$ with respect to a module $K$ satisfying $\NF(M)\cap\NF(N)\subseteq\Y^n(K)$ with $n=\depth_RK$, assuming that $K$ is semidualizing for (2).
We leave it to the reader as an exercise.
\end{rem}

Corollary \ref{12} immediately recovers Corollary \ref{ard} in the case where $\dim R\ge2$, as follows.
Thus our Theorem \ref{7} (Theorem \ref{s}) recovers the Auslander--Reiten duality theorem in dimension at least two.

\begin{cor}[Auslander--Reiten duality in dimension at least two]\label{ards}
Let $(R,\m)$ be a Gorenstein local ring with $d\ge2$.
Let $M,N$ be maximal Cohen--Macaulay $R$-modules with $\NF(M)\cap\NF(N)\subseteq\{\m\}$.
Then for each $i\in\Z$ there is an isomorphism
$$
\cext_R^i(N,M)^\vee\cong\cext_R^{(d-1)-i}(M,N).
$$

\end{cor}

\begin{proof}
As $d\ge2$, one can apply Corollary \ref{12}(2) to get $\Ext^{d-1}_R(N,M)\cong\Ext^d_R(\lhom_R(M,N),R)$, which and \cite[Corollary 3.5.9]{BH} show $\Ext^{d-1}_R(N,M)^\vee\cong\lhom_R(M,N)$.
Substituting $\syz^{i-(d-1)}N$ for $N$ (and using basic facts on Tate cohomology \cite[Section 7]{catgp}), we obtain
$$
\cext^i_R(N,M)^\vee\cong\Ext^{d-1}_R(\syz^{i-(d-1)}N,M)^\vee\cong\lhom_R(M,\syz^{i-(d-1)}N)\cong\cext^{(d-1)-i}_R(M,N),
$$
which give an isomorphism as in the assertion.
\end{proof}

We prepare a lemma for the proof of the second main result of this paper, which is also be used later.
The proof of the lemma is analogous to that of \cite[Lemma 2.3]{GT}. 

\begin{lem}\label{l1}
Let $R$ be a Cohen--Macaulay local ring with canonical module $\omega$.
Let $M$ be an $R$-module and $N$ a maximal Cohen--Macaulay $R$-module.
Let $0\le m\le d$ be an integer such that $\NF(M)\cap\NF(N)\subseteq\Y^m(R)$.
Then $\Ext^i_R(M,N^\dagger)\cong\Ext^i_R(M\otimes_RN,\omega)$ for all integers $0\leq i\leq m$.	
\end{lem}

We are now ready to present our second main result, which includes Theorem \ref{t} from the Introduction.

\begin{thm}\label{14}
Let $(R,\m)$ be a Cohen--Macaulay local ring with canonical module $\omega$.
Let $M,N$ be $R$-modules such that $\NF(M)\cap\NF(N)\subseteq\{\m\}$. Assume that $N$ is maximal Cohen--Macaulay.
\begin{enumerate}[\rm(1)]
\item There is an isomorphism $\lhom_R(M,N)^\vee\cong\Ext_R^{d+1}(\Tr M,N^\dag)$.
Therefore, if $d\ge2$, then one has
$$
\lhom_R(M,N)^\vee\cong\Ext_R^{d-1}(M^*,N^\dag).
$$
\item
Suppose that $M$ is locally totally reflexive on the punctured spectrum of $R$.
Then there is an isomorphism $\Ext_R^i(M,N)^\vee\cong\Ext_R^{d+1-i}(\Tr M,N^\dag)$ for all $1\le i\le d$.
In particular, for all $1\leq i\leq d-2$ one has
$$
\Ext_R^i(M,N)^\vee\cong\Ext_R^{(d-1)-i}(M^*,N^\dag).
$$
\end{enumerate}
\end{thm}

\begin{proof}
We claim that $E:=\Ext^j_R(\Tr M,N^\dag)$ has finite length for all $j\ge d$.
Indeed, let $\p$ be a non-maximal prime ideal.
As $\NF(M)\cap\NF(N)\subseteq\{\m\}$, either $M_\p$ or $N_\p$ is free.
If $M_\p$ is free, then so is $\Tr M_\p$, and $E_\p=0$.
If $N_\p$ is free, then $(N^\dag)_\p$ is a direct sum of copies of $\omega_\p$, which has injective dimension at most $\dim R_\p<d$, and hence $E_\p=0$.

(1) By assumption, $\Tor_i^R(\Tr M,N)$ has finite length for all $i>0$.
There are isomorphisms
$$
\lhom_R(M,N)
\cong\Tor_1^R(\Tr M,N)
\cong\Ext^d_R(\Tor_1(\Tr M,N),\omega)^\vee
\cong\Ext^{d+1}_R(\Tr M,N^\dag)^\vee,
$$
where the first two isomorphisms are obtained by \cite[Lemma (3.9)]{Yo} and \cite[Corollary 3.5.9]{BH}, respectively, while the last isomorphism is shown to hold by the spectral sequence argument in \cite[2.4]{ACST}.
By the above claim we see that $\lhom_R(M,N)^\vee\cong\Ext^{d+1}_R(\Tr M,N^\dag)^{\vee\vee}\cong\Ext^{d+1}_R(\Tr M,N^\dag)$ by \cite[Proposition 3.2.12(c)]{BH}.

(2) We show the assertion by induction on $i$.
There are isomorphisms
$$
\Ext_R^1(M,N)
\cong\Gamma_\m(\Ext_R^1(M,N))
\cong\Gamma_\m(\Tr M\otimes_RN)
\cong\Ext_R^d(\Tr M\otimes_RN,\omega)^\vee
\cong\Ext_R^d(\Tr M,N^\dag)^\vee.
$$
Let us explain each of the above isomorphisms.
As $\NF(M)\cap\NF(N)\subseteq\{\m\}$ and $M$ is locally totally reflexive on the punctured spectrum, $\Ext^1_R(M,N)$ has finite length.
Hence the first isomorphism holds.
There is an exact sequence $0\to\Ext^1_R(M,N)\to\Tr M\otimes_RN\to\Hom_R((\Tr M)^*,N)$ by \cite[Proposition (2.6)]{AB}, while $\Hom_R((\Tr M)^*,N)$ has positive depth by \cite[Exercise 1.4.19]{BH}.
Thus the second isomorphism holds.
The third and fourth isomorphisms follow from \cite[Corollary 3.5.9]{BH} and Lemma \ref{l1}, respectively.
Consequently, $\Ext_R^1(M,N)^\vee$ is isomorphic to $\Ext_R^d(\Tr M,N^\dag)^{\vee\vee}$, which is isomorphic to $\Ext_R^d(\Tr M,N^\dag)$ by the claim given at the beginning of the proof and \cite[Proposition 3.2.12(c)]{BH}.
Thus the case $i=1$ is settled.

Let $i\ge2$.
Applying the induction hypothesis to $\syz M$, we obtain isomorphisms
$$
\Ext_R^i(M,N)^\vee\cong\Ext_R^{i-1}(\syz M,N)^\vee\cong\Ext_R^{d+1-(i-1)}(\Tr\syz M,N^\dag)\cong\Ext_R^{d+1-i}(\syz\Tr\syz M,N^\dag).
$$
There is an exact sequence $0\to\Ext_R^1(M,R)\to\Tr M\to\syz\Tr\syz M\to0$ by \cite[Proposition (2.6)]{AB}.
Since $\Ext_R^1(M,R)$ has finite length and $N^\dag$ is maximal Cohen--Macaulay, $\Ext_R^j(\Ext_R^1(M,R),N^\dag)=0$ for all $j<d$.
We obtain an isomorphism $\Ext_R^{d+1-i}(\syz\Tr\syz M,N^\dag)\cong\Ext_R^{d+1-i}(\Tr M,N^\dag)$, which completes the proof.
\end{proof}

Theorem \ref{14} not only recovers but also refines a main result of \cite{CT} (more precisely, \cite[Theorem 1.1]{CT}), where $M$ is assumed to be locally free on the punctured spectrum.

\begin{cor}[Celikbas--Takahashi, improved]\label{ct}
Let $(R,\m)$ be a Cohen--Macaulay local ring with a canonical module.
Let $X$ be a totally reflexive $R$-module, and let $M$ be a maximal Cohen--Macaulay $R$-module.
Assume that $\NF(X)\cap\NF(M)\subseteq\{\m\}$.
Then for each $i\in\Z$ there is an isomorphism
$$
\cext_R^i(X,M)^\vee\cong\cext_R^{(d-1)-i}(X^\ast,M^\dag).
$$
\end{cor}

\begin{proof}
We have the following isomorphisms, where the second one follows from Theorem \ref{14}(1) and the fact that the Ext module has finite length (use \cite[Proposition 3.2.12(c)]{BH}).
\begin{align*}
\cext^i_R(X,M)^\vee
&\cong\lhom_R(\syz^i X,M)^\vee
\cong\Ext^{d+1}_R(\Tr\syz^iX,M^\dag)
\cong\lhom_R(\syz^{d+1}\Tr\syz^iX,M^\dag),\\
\cext^{(d-1)-i}_R(X^\ast,M^\dag)
&\cong\lhom_R(\syz^{(d-1)-i}\syz^2\Tr X,M^\dag)
\cong\lhom_R(\syz^{d+1-i}\Tr X,M^\dag).
\end{align*}
It suffices to show $\Tr\syz^iX\cong\syz^{-i}\Tr X$ up to free summands, which follows from \cite[Proposition 7.1]{catgp}.
\end{proof}

It is evident that Corollary \ref{ct} implies Corollary \ref{ard}.
Thus our Theorem \ref{14} (Theorem \ref{t}) also recovers the Auslander--Reiten duality theorem (for any dimension).

%%%%%%%%%%%%%%%%%%%%%%%%%%%%%%%%%%%%%%%%%%%%%%%%%%%%%%%%%%%%%%
\section{Testing freeness by vanishing of Ext modules}

In this section, we present applications of our results obtained in the previous section to give criteria for freeness of modules over local rings and to recover various known results about the Auslander--Reiten conjecture.
To give our first application, we establish a lemma, which will also be used later.

\begin{lem}\label{22}
Let $M$, $N$ be $R$-modules.
Suppose that $M$ is reflexive (i.e. $2$-torsionfree) and that $\NF(M)\cap\NF(N)\subseteq\Y^2(R)$.
Then one has an isomorphism $\Hom_R(M,N)^*\cong\Hom_R(N,M)$.
\end{lem}

\begin{proof}
Let $n=2$, $i=0$ and $K=R$ in Lemma \ref{21}(2), and then use the tensor-hom adjunction.
\end{proof}

\begin{cor}\label{5}
Let $R$ be a local ring, and let $M$ be a $(t+1)$st syzygy $R$-module.
Assume that $M$ is locally free on $\X^n(R)$ for some $0<n<t$.
Then $M$ is free, provided that one of the following three conditions is satisfied:
\begin{enumerate}[\rm(1)]
\item
$\Ext_R^i(\Hom_R(M,M),R)=\Ext_R^j(M,M)=0$ for all $n+1\le i\le t$ and $n\le j\le t-1$.
\item
$\Ext_R^j(M,M)=0$ for all $n\le j\le t-1$ and $\gdim_R\Hom_R(M,M)<\infty$.
\item
$\Ext_R^j(M,M)=0$ for all $n\le j\le t-1$ and $\Hom_R(M,M)$ is $(t+2)$-torsionfree.
\end{enumerate}
\end{cor}

\begin{proof}
First of all, note that the inequalities $0<n<t$ especially say that $t\ge2$.

(1) We use induction on $s=t-n$.
If $s=1$, then $n=t-1$.
Note from Example \ref{15}(3) that $M$ is $(t+1)$-torsionfree.
Corollary \ref{12}(1) implies $\lhom_R(M,M)=0$, which is equivalent to saying that $M$ is free.
Let $s\ge2$.

We claim that $M$ is locally free on $\X^{t-1}(R)$.
Indeed, pick any prime ideal $\p\in\X^{t-1}(R)$ and set $t'=\depth R_\p$.
If $t'\le n$, then $\p\in\X^n(R)$, and $M_\p$ is $R_\p$-free by assumption.
If $t'>n$, then $0<n<t'$ and $M_\p$ is locally free on $\X^n(R_\p)$.
As $t'\le t$, the module $M_\p$ is a $(t'+1)$st syzygy $R_\p$-module, and $\Ext_{R_\p}^i(\Hom_{R_\p}(M_\p,M_\p),R_\p)=\Ext_{R_\p}^j(M_\p,M_\p)=0$ for all $n+1\le i\le t'$ and $n\le j\le t'-1$.
Since $t'-n<s$, we can apply the induction hypothesis to deduce that $M_\p$ is $R_\p$-free.
Thus the claim follows.

By Example \ref{15}(3) the module $M$ is $(t+1)$-torsionfree.
Corollary \ref{12}(1) shows $\lhom_R(M,M)=0$, so $M$ is free.
	
(2) Similarly to the proof of (1), one can prove that $M$ is $(t+1)$-torsionfree and locally free on $\X^{t-1}(R)$.
As $M$ is a $(t+1)$st syzygy module and $t\ge2$, the depth lemma shows that $M$ has depth at least two, and so does $\Hom_R(M,M)$ by \cite[Exercise 1.4.19]{BH}.
Now the assertion follows from Corollary \ref{12}(2).

(3) Again, one can show that $M$ is $(t+1)$-torsionfree.
In particular, $M$ is $2$-torsionfree.
Set $H=\Hom_R(M,M)$.
Lemma \ref{22} implies $H\cong H^*$.
As $H$ is $(t+2)$-torsionfree, we have $\Ext^i_R(H,R)\cong\Ext^i_R(H^*,R)\cong\Ext^{i+2}_R(\Tr H,R)=0$ for $1\leq i\leq t$.
Now the assertion follows from (1).
\end{proof}

We record various consequences of Corollary \ref{5}.
First of all, Corollary \ref{5}(2) immediately recovers the theorems of Araya \cite[Corollary 10]{A}, Ono and Yoshino \cite[Theorem]{OnYos}, and the following result due to Araya, Celikbas, Sadeghi and Takahashi \cite[Corollary 1.6]{ACST}.

\begin{cor}[Araya--Celikbas--Sadeghi--Takahashi]
Let $R$ be a Gorenstein local ring.
Let $M$ be a maximal Cohen--Macaulay $R$-module, and $0<n<d$ an integer.
Suppose that $M$ is locally free on $\X^n(R)$ and $\Ext_R^i(M,M)=0$ for all $n\le i\le d-1$.
Then $M$ is free.
\end{cor}

Recall that an $R$-module $M$ is called {\em torsion-free} if the natural map $M\to M\otimes_R\Q(R)$ is injective, or in other words, each non-zerodivisor of $R$ is also a non-zerodivisor of $M$.

\begin{cor}\label{l}
Let $R$ be a Cohen--Macaulay local ring of dimension $d$ with a canonical module $\omega$.
Assume that $M$ is a $(d+1)$st syzygy $R$-module that is locally free in codimension one.
Suppose that $\Ext^i_R(M,M)=0$ for $1\leq i\leq d-1$ and $\Hom_R(M,M)\otimes_R\omega_R$ is torsion-free.
Then $M$ is free.
\end{cor}

\begin{proof}
First of all, since $R$ satisfies $(\S_1)$, an $R$-module is torsion-free if and only if it satisfies $(\S_1)$.
Hence the torsion-free property localizes.
We argue by induction on $d$.
The assertion is evident if $d\leq1$, so we may assume that $d>1$ and $M$ is locally free on the punctured spectrum of $R$.
Set $H=\Hom_R(M,M)$.
As $H$ is locally free on the punctured spectrum, by Lemma \ref{l1} we have $\Ext^i_R(\omega\otimes_RH,\omega)\cong\Ext^i_R(H,R)$ for all $0\le i\le d$.
Since $\omega\otimes_RH$ has positive depth, $\Ext_R^d(H,R)=0$ by local duality.
The assertion is deduced from Corollary \ref{5}(1).	
\end{proof}

The following two results are also consequences of Corollary \ref{5}.
The first one is nothing but Corollary \ref{cc} from the Introduction.
Both should be compared with the theorems of Huneke and Leuschke \cite[Theorem 1.3]{HL}, Goto and Takahashi \cite[Theorem 1.5(2)]{GT}, and Dao, Eghbali and Lyle \cite[Theorem 3.16]{DEL}.

\begin{cor}\label{c3}
Let $R$ be a generically Gorenstein local ring, and let $0<n<t$ be an integer.
Let $M$ be an $R$-module which is locally free on $\X^n(R)$.
Then $M$ is free, if the following three conditions hold.
\begin{enumerate}[\rm(1)]
\item
$\Ext^i_R(M,M)=0$ for all $n\leq i\leq t-1$.
\item
$\Ext^i_R(M,R)=0$ for all $i>0$.
\item $\Ext^i_R(\Hom_R(M,M),R)=0$ for all $n+1\leq i\leq t$, or $\Hom_R(M,M)$ has finite G-dimension.
\end{enumerate}
\end{cor}

\begin{proof}
It follows from generic Gorensteinness, (2) and \cite[Corollary 1.3]{Y} that $M$ is a totally reflexive module, and a $j$th syzygy module for all $j\ge0$.
The assertion follows from the first and second assertions of Corollary \ref{5}.
\end{proof}

Note that every normal ring $R$ is generically Gorenstein and every syzygy $R$-module is locally free on $\X^1(R)$.
Thus the following is immediate from Corollary \ref{c3}, which is nothing but Corollary \ref{abc} from the Introduction.

\begin{cor}\label{ccc}
Let $R$ be a normal local ring.
An $R$-module $M$ is free if either
\begin{enumerate}[\rm(1)]
\item
$\Ext^i_R(M,M)=\Ext^j_R(M,R)=\Ext^h_R(\Hom_R(M,M),R)=0$ for all $1\le i\le t-1$, $j\ge1$ and $2\le h\le t$, or
\item
$\Ext^i_R(M,M)=\Ext^j_R(M,R)=0$ for all $1\le i\le t-1$ and $j\ge1$, and that $\Hom_R(M,M)$ has finite G-dimension.
\end{enumerate}
\end{cor}

Corollary \ref{ccc} recovers the following result due to Huneke and Leuschke \cite[Theorem 0.1]{HL} and Araya \cite[Corollary 4]{A}, which asserts that the Auslander--Reiten conjecture holds for normal Gorenstein rings.

\begin{cor}[Huneke--Leuschke, Araya]
Let $R$ be a (not necessarily local) Gorenstein normal ring.
Let $M$ be an $R$-module.
If $\Ext_R^i(M,M)=0=\Ext_R^i(M,R)$ for all $i>0$, then $M$ is projective.
\end{cor}

Now we deal with reflexive modules $M$ such that $\Hom_R(M,M)$ is a direct sum of copies of $M$.
Such modules were first studied by Auslander \cite{A3} to give a new proof of a theorem on the purity of the branch locus for regular local rings; see \cite[Proposition 1.2 and Theorem 1.3]{A3}.
The following result especially says that the Auslander--Reiten conjecture holds true for such modules over normal rings.

\begin{cor}\label{c}
Let R be a local ring with $t\ge2$, and let $M$ be a reflexive $R$-module that is locally free on $\X^1(R)$.
Assume that $\Hom_R(M,M)\cong M^{\oplus n}$ for some $n>0$.
If $\Ext^i_R(M,R)=\Ext_R^j(M,M)=0$ for all $1\le i\le t$ and $1\le j\le t-1$, then $M$ is free.
\end{cor}	

\begin{proof}
We use induction on $t$.
We may assume that $M$ is locally free on $\X^{t-1}(R)$.
By Lemma \ref{22}, we have $M^{\oplus n}\cong\Hom_R(M,M)\cong\Hom_R(M,M)^*\cong(M^*)^{\oplus n}$.
Since $M$ is reflexive, so is $\Hom_R(M,M)$ and
$$
\Ext_R^i(\Tr\Hom_R(M,M),R)
\cong\Ext_R^{i-2}(\Hom_R(M,M)^*,R)
\cong\Ext_R^{i-2}(M,R)^{\oplus n}
=0
$$
for all $3\le i\le t+2$.
Hence $\Hom_R(M,M)$ is $(t+2)$-torsionfree, and so is $M$.
In particular, $M$ is a $(t+2)$nd syzygy module.
Now the assertion follows from Corollary \ref{5}(3).
\end{proof}

Here we state an application of Theorem \ref{14}.

\begin{cor}\label{t2}
Let $R$ be a Cohen--Macaulay local ring with canonical module $\omega$.
Let $0<n<d$ be an integer, and let $M$ be an $R$-module that is locally free on $\X^n(R)$.
Then $M$ is free, if it satisfies one of the following conditions.
\begin{enumerate}[\rm(1)]
\item
$M$ satisfies $(\S_1)$, $M^*$ is maximal Cohen--Macaulay, and $\Ext_R^i(M,(M^*)^\dagger)=0$ for all $n\le i<d$.
\item
$\Min M\subseteq\Ass R$, $M^*$ is maximal Cohen--Macaulay, and $\Ext_R^i(M,(M^*)^\dagger)=0$ for all $n\le i\le d$.
\item
$M$ is a syzygy module, and $\Ext_R^i(M,R)=\Ext_R^j(M,M\otimes_R\omega)=0$ for all $0<i\le d$ and $n\le j<d$.
\end{enumerate}
\end{cor}

\begin{proof}
(1) First of all, note that the inequalities $0<n<d$ especially says $d\ge2$.

(i) We begin with showing the assertion in the case where $M$ is locally free on the punctured spectrum of $R$.
In this case, Theorem \ref{14}(1) implies $\lhom_R(M^*,M^*)^\vee\cong\Ext_R^{d-1}(M^{**},M^{*\dag})$.
Let $f:M\to M^{**}$ be the natural homomorphism.
As $f$ is locally an isomorphism on the punctured spectrum, $K=\Ker f$ and $C=\Cok f$ have finite length.
Since $M^{*\dag}$ is maximal Cohen--Macaulay, we have $\Ext_R^j(K,M^{*\dag})=0=\Ext_R^j(C,M^{*\dag})=0$ for all $j<d$.
It is easy to observe that the homomorphism $\Ext_R^{d-1}(M^{**},M^{*\dag})\to\Ext_R^{d-1}(M,M^{*\dag})$ induced from $f$ is injective.
By assumption, the latter Ext module vanishes.
Hence $\lhom_R(M^*,M^*)=0$, which implies that $M^*$ is free.

Consider the case $M^*=0$.
Then $f=0$ and $M=K$.
Hence $M$ has finite length, while $M$ has positive depth as it satisfies $(\S_1)$.
Therefore $M=0$.
In particular $M$ is free, and we are done.
Thus we may assume $M^*\ne0$.
Theorem \ref{14}(2) implies $\Ext_R^{d-1}(M,M^{*\dag})^\vee\cong\Ext_R^2(\Tr M,M^*)$, and the former Ext module vanishes.
As $M^*$ is a nonzero free module, $\Ext_R^2(\Tr M,R)=0$.
This means $C=0$, and we have an exact sequence $0\to K\to M\to M^{**}\to0$, whose last term is free.
Hence this exact sequence splits, and we get an isomorphism $M\cong K\oplus M^{**}$.
As $M$ has positive depth and $K$ has finite length, we have $K=0$.
Consequently, $M$ is free.

(ii) Now let us consider the general case.
We handle this by induction on $d$.
When $d=2$, we have $n=1=d-1$, and $\X^n(R)$ coincides with the punctured spectrum of $R$.
Case (i) shows the assertion.
Let $d>2$, and pick any non-maximal prime ideal $\p$.
Then one can apply the induction hypothesis to the localizations $R_\p$ and $M_\p$ to deduce that $M_\p$ is free.
This means that $M$ is locally free on the punctured spectrum, and again case (i) implies the assertion.

(2) We use the same argument as the proof of (1).
Using the assumption that $\Min M\subseteq\Ass R$, we easily deduce that $M^*_\p\ne0$ for all $\p\in\supp M$.
When $M$ is locally free on the punctured spectrum, $M^*$ is a free module, and It follows from Theorem \ref{14}(2) that $\Ext_R^1(\Tr M,M^*)\cong\Ext_R^d(M,M^{*\dag})^\vee=0$.
As $M^*\ne0$, we have $\Ext_R^1(\Tr M,R)=0$.
This means that $M$ is a syzygy, and in particular it satisfies $(\S_1)$.

(3) Dualizing a free resolution $F$ of $M$ by $R$, we obtain an exact sequence $0\to M^*\to (F_0)^*\to\cdots\to (F_{d+1})^*$.
This especially says that $M^*$ is maximal Cohen--Macaulay.
By \cite[Corollary B4(3)]{ABS} we have $(M^*)^\dagger\cong M\otimes_R\omega$.
As $M$ is a syzygy module, it satisfies $(\S_1)$.
Now the assertion follows from (1).
\end{proof}

The above result is actually a refinement of \cite[Theorem 1.4 and Corollary 2.7]{ACST}.

\begin{cor}[Araya--Celikbas--Sadeghi--Takahashi]
Let $R,\omega,n$ be as in Corollary \ref{t2}.
\begin{enumerate}[\rm(1)]
\item
An $R$-module $M$ is free if it satisfies the following conditions.
\begin{enumerate}[\rm(i)]
\item
$M$ is locally of finite projective dimension on $\X^{n}(R)$.
\item
$M$ satisfies $(\S_2)$ and $M^{\ast}$ is maximal Cohen--Macaulay.
\item
$\Ext^{i}_R(M,(M^{\ast})^{\dagger})=0$ for all $n\le i<d$.
\end{enumerate}
\item
An $R$-module $M$ is free if it satisfies the following conditions.
\begin{enumerate}[\rm(i)]
\item
$M$ is locally of finite projective dimension on $\X^{n}(R)$.
\item
$M$ is reflexive and $\Ext^i_R(M,R)=0$ for all $1\le i\le d$.
\item
$\Ext^{i}_R(M,M\otimes_R\omega)=0$ for all $n\le i<d$.
\end{enumerate}
\end{enumerate}
\end{cor}

\begin{proof}
(1) Similarly to the beginning of the proof of \cite[Theorem 1.4]{ACST}, the module $M$ is locally free on $\X^n(R)$.
Corollary \ref{t2}(1) implies the assertion.

(2) Since $M$ satisfies $(\S_2)$, it is reflexive.
As in the proof of Corollary \ref{t2}(3) the module $M^*$ is maximal Cohen--Macaulay.
Thus the assertion follows from (1).
\end{proof}

From here to the end of this section, we consider a {\em rational normal surface singularity}, that is, a complete local normal domain of dimension two over an algebraically closed field of characteristic zero which has a rational singularity.
Recall that an $R$-module $M$ is called {\em rigid} if $\Ext_R^1(M,M)=0$.
The following corollary is deduced from our Theorem \ref{7} and Corollary \ref{5}.

\begin{cor}
Let $R$ be a rational normal surface singularity.
Let $M$ be a third syzygy $R$-module.
If $M$ is rigid, then $M$ is free.
\end{cor}

\begin{proof}
Set $H=\Hom_R(M,M)$.
Note that $H$ is maximal Cohen--Macaulay (see \cite[Exercise 1.4.19]{BH}).
The module $M$ is $3$-torsionfree by Example \ref{15}(3).
It follows from Theorem \ref{7}(1) and the rigidity of $M$ that $\Ext_R^1(H,R)=0$.
Lemma \ref{22} implies $H\cong H^*$.
By \cite[Theorem 2.2 and Proposition 3.1(2)]{IW} the module $H$ is free.
Finally, our Corollary \ref{5}(1) deduces that $M$ is free.
\end{proof}

In view of Corollary \ref{5}(3), we are interested in when a given module is $(d+2)$-torsionfree.
Any totally reflexive module is $(d+2)$-torsionfree.
If $R$ is Gorenstein, then it is equivalent to maximal Cohen--Macaulayness.
Thus we are interested in when a non-totally reflexive module over a non-Gorenstein ring is $(d+2)$-torsionfree.
We discuss the triviality of such modules.

\begin{prop}\label{13}
Let $R$ be a non-Gorenstein complete local ring with $d\ge2$ such that $k$ is algebraically closed and has characteristic $0$.
Assume that there exists an $R$-sequence $\xx=x_1,\dots,x_{d-2}$ such that $R/(\xx)$ is a rational normal surface singularity.
Let $M$ be an $R$-module such that $\Ext_R^i(M,R)=0$ for all $1\le i\le 2d$.
Then $M$ is $R$-free.
\end{prop}

\begin{proof}
Set $N=\syz_R^{d-2}M$.
It is seen from \cite[Proposition 1.1.1]{I} that $N$ is a $(d-2)$-torsionfree module with $\Ext_R^i(N,R)=0$ for $1\le i\le d+2$.
As $N$ is a $(d-2)$nd syzygy module, $\xx$ is an $N$-sequence.
Hence the Koszul complex of $\xx$ on $N$ induces an exact sequence
$$
0 \to N \to N^{\oplus(d-2)} \to \cdots \to N^{\oplus(d-2)} \to N \to N/\xx N \to 0,
$$
which shows that $\Ext_{R/(\xx)}^i(N/\xx N,R/(\xx))\cong\Ext_R^{i+d-2}(N/\xx N,R)=0$ for $1\le i\le4$ (see \cite[Lemma 3.1.16]{BH}).
It follows from \cite[Proposition 3.1(1)]{IW} and \cite[Proposition 1.1.1]{I} again that $\syz^2_{R/(\xx)}(N/\xx N)$ is free over $R/(\xx)$.
Thus $\pd_RN=\pd_{R/(\xx)}N/\xx N$ is finite (see \cite[Lemma 1.3.5]{BH}), and so is $\pd_RM$.
The vanishing assumption on Ext modules forces $M$ to be free over $R$.
\end{proof}

\begin{cor}
Let $R$ be as in Proposition \ref{13}.
Let $M$ be a $(d+2)$-torsionfree $R$-module such that $\Ext_R^i(M,R)=0$ for all $1\le i\le d-2$.
Then $M$ is free.
In particular:
\begin{enumerate}[\rm(1)]
\item
When $d=2$, every $(d+2)$-torsionfree $R$-module is free.
\item
$R$ is a G-regular local ring in the sense of \cite{greg}.
\end{enumerate}
\end{cor}

\begin{proof}
We observe that $N:=\Tr\syz^{d-2}M$ satisfies $\Ext_R^i(N,R)=0$ for $1\le i\le 2d$; see \cite[Proposition 1.1.1]{I}.
Proposition \ref{13} implies that $N$ is free, and hence so is $\syz^{d-2}M$.
Thus $M$ has finite projective dimension.
Since $M$ is $(d+2)$-torsionfree, it is maximal Cohen--Macaulay.
Therefore $M$ is free.
\end{proof}

%%%%%%%%%%%%%%%%%%%%%%%%%%%%%%%%%%%%%%%%%%%%%%%%%%%%%%%%%%%%%%%%%%%

%%%%%%%%%%%%%%%%%%%%%%%%%%%%%%%%%%%%%%%%%%%%%%%%%%%%%%%%%%%%%%%%%%%
\end{document}